\documentclass[12pt, reqno]{amsart}
\usepackage{amsmath, amsthm, amscd, amsfonts, amssymb, graphicx, color}
\usepackage[bookmarksnumbered, colorlinks, plainpages]{hyperref}

\newtheorem{theorem}{Theorem}[section]
\newtheorem{lemma}[theorem]{Lemma}
\newtheorem{proposition}[theorem]{Proposition}
\newtheorem{corollary}[theorem]{Corollary}
\theoremstyle{definition}

\numberwithin{equation}{section}

 \def\buildrel#1_#2^#3{\mathrel{\mathop{\kern 0pt#1}\limits_{#2}^{#3}}}

  \def\K{{ \mathbb K}}
    \def\R{{ \mathbb R}}
 \def\C{{ \mathbb C}}

 \def\N{{ \mathbb N}}

 \def\e{\varepsilon}
   \def\bs{\boldsymbol}

 \def\inter{\cap}
 
 \def\ov{\overline}
  \def\ss{\subseteq}
 \def\emp{\emptyset}

\begin{document}
\setcounter{page}{1}

\title[Logarithms and exponentials in Banach algebras]
{Logarithms and exponentials in Banach algebras}

\author[R. Mortini, R. Rupp]{Raymond Mortini$^1$$^{*}$ and Rudolf Rupp$^2$}

\address{$^{1}$ Universit\'{e} de Lorraine,
 D\'{e}partement de Math\'{e}matiques et  
Institut \'Elie Cartan de Lorraine,  UMR 7502,
 Ile du Saulcy,
 F-57045 Metz, France}
\email{\textcolor[rgb]{0.00,0.00,0.84}{raymond.mortini@univ-lorraine.fr}}

\address{$^{2}$  Fakult\"at f\"ur Angewandte Mathematik, Physik  und Allgemeinwissenschaften,
 TH-N\"urnberg,
 Kesslerplatz 12, D-90489 N\"urnberg, Germany}
\email{\textcolor[rgb]{0.00,0.00,0.84}{Rudolf.Rupp@th-nuernberg.de}}



\date{Received: xxxxxx; Revised: yyyyyy; Accepted: zzzzzz.
\newline \indent $^{*}$ Corresponding author}

\begin{abstract}
Let $A$  be a complex Banach algebra.  If the   spectrum of an invertible element $a\in A$  
does not separate the plane, then $a$ admits a logarithm. We present two elementary 
proofs of this classical result 
 which are independent of the holomorphic functional calculus.   We also discuss
the case of   real Banach algebras.  As applications, we  obtain simple proofs  that every 
invertible matrix over $\C$ has a logarithm  and that every real matrix 
$M$ in $M_n(\R)$ with $\det M>0$  is a product of two  real
exponential matrices. 
\end{abstract} \maketitle

 \centerline {\small\the\day.\the \month.\the\year} \medskip

\section{Complex Banach algebras}
  The purpose of this note is to show  that the  following important  result in the realm of 
  Banach algebras   can also be proven by
 elementary methods which are independent of the holomorphic functional calculus.
  For the classical  proof, see for  instance
  \cite[Theorem 10.30, p. 264] {rud}.
  
   \begin{theorem}\label{expoinba}
Let $A$ be a unital complex Banach algebra (not necessarily commutative) and let $a\in A$.
Suppose that $0$ belongs to the unbounded connected component of $\C\setminus \sigma_A( a)$,
where  $\sigma_A( a)$ is the spectrum
of $a$. Then there exists $b\in A$ such that $e^b=a$.
\end{theorem}

  We shall only use the following standard result which can be found in almost all
 monographs on functional analysis and Banach algebras.
   \begin{lemma}\label{expo}
      Let $A=(A,||\cdot||)$  be a unital, real or complex,  but not necessarily commutative
      Banach algebra, the norm $||\cdot||$ being submultiplicative and let
      $$\exp A=\{e^a: a\in A\}.$$
   Suppose that  $f\in A$ satisfies
       $||{\bs 1}-f||<1$. Then $f\in \exp A$.  Moreover, 
     if  $L$ is given by
      $$L=-\sum_{j=1}^\infty \frac{1}{j}\big({\bs 1}-f\big)^{j},$$
    then   the series is unconditionnally/absolutely  convergent in $A$ and $e^L=f$.
    Finally, if for some $g\in A$,  $M>||g||$, then  
    $$M\cdot  {\bs 1}-g\in \exp A  \;\footnote{ Note that in the real case we do not necessarily have
   that the opposite $g-M\cdot \bs 1$ belongs to $\exp A$.},$$
  and in the case of a complex algebra,  
     $$g-\lambda\cdot{\bs 1}\in \exp A$$
      for every $\lambda\in \C$ with $|\lambda|>||g||$.
      \end{lemma}                                                
\begin{proof}
 For the reader's convenience, we  prove the latter assertion.  Observe that 
 $$ h:=M\cdot  {\bs 1}-g= M\big(\bs 1-(g/M)\big).$$
 Now apply the first assertion to  $f:=\bs 1- (g/M)$. Then $f=e^a$ for some $a\in A$. Let
  $m:=\log M$ be  a real logarithm of $M$.
 Since $m\cdot{\bs 1}$ commutes with each
 $a\in A$, we deduce that 
 $$h= e^{m\cdot{\bs 1}}e^a= e^{m\cdot{\bs 1}+a}\in \exp A.$$
 If $\K=\C$, we let  $m$ be a fixed complex logarithm of $\lambda$. Then
 $$g-\lambda\cdot{\bs 1}=-(\lambda \cdot{\bs1}-g)= -e^{m\cdot {\bs 1}+a}=
 e^{(i\pi+m)\cdot \bs 1 +a}\in \exp A. \eqno\qedhere$$
\end{proof}
 \bigskip

{\bf First proof of Theorem \ref{expoinba}.}

Consider the closed subalgebra $B:=[a]_{{\rm alg}}$ generated by $a$ in $A$.
Then $B$ is a commutative complex Banach algebra.   Note that 
 $\sigma_B(a)$ is contained in the polynomial convex hull of $\sigma_A(a)$
  (see \cite[Theorem 10.18 (b), p. 256]{rud})
 \footnote{ We actually have that $\sigma_B(a)=\widehat{\sigma_A(a)}$.}. In particular, 
$0$ stays  in the unbounded component of $\C\setminus \sigma_B(a)$.
Hence, there is an arc $J$ in $\C\setminus \sigma_B(a)$ given by the parametrization  
$\varphi: [0,1[\to \C\setminus  \sigma_B(a)$
 joining $0$ to infinity.  Since $\Phi(t):=a -\varphi(t)\cdot \bs 1\in B^{-1}$ for every $t\in [0,1[$,
 and $\Phi(t)\in \exp B$  for all $t$ sufficiently close to $1$ (Lemma \ref{expo}),
  the fact that
 $\exp B$ is a maximal connected subset of $B^{-1}$ now  implies that $\Phi(t)\in \exp B$ for every $t$. In particular,  $\Phi(0)=a\in \exp B$.
 Because $B\ss A$, we have found $b\in A$ such that $e^b=a$.
\hfill $\square$
 \medskip
 
 Yet an even more elementary proof can be given along the following lines. 
 It  avoids the necessity to know that
 the complement of the spectrum   of $a$ with respect to the algebra $B=[a]_{{\rm alg}}$ has the same unbounded  component  as the   complement of the spectrum   of $a$ with respect to the `big'  algebra $A$.  This second proof only makes use of Lemma \ref{expo}  and a simple connectedness argument.
It is suitable for an undergraduate course in analysis and linear algebra. \bigskip
 
{\bf Second proof of Theorem \ref{expoinba}.}

Let  $J'$ be an  arc in $\C\setminus \sigma_A(a)$ joining $0$ to infinity  and 
 which is given by the parametrization  
$\varphi: [0,1[\to \C\setminus  \sigma_A(a)$. Let us emphasize that in contrast to the first proof, 
 the arc lives in $\C\setminus \sigma_A(a)$. As above,  let 
 $B:=[a]_{{\rm alg}}$ be the closed subalgebra of $A$ generated by $a$.
Associate with $\varphi$ the path $\Phi:[0,1[\to A^{-1}\inter B$
given by 
$$\Phi(t):=a -\varphi(t)\cdot \bs 1,$$
and consider the set
$$\mathcal O=\{t\in [0,1[:  \Phi(t)\in \exp B\}.$$
Since  by Lemma \ref{expo},  $\Phi(t)\in \exp B$  for $t$ close to $1$, $\mathcal O\not=\emp$. \\

{\bf Claim 1} $\mathcal O$ is open in $[0,1[$. 

To see this, let  $t_1\in \mathcal O$. Then $\Phi(t_1)=e^h$ for some $h\in B$. Let 
$0<\e< ||e^{-h}||^{-1}$ and choose a neighborhood $I_\delta\ss [0,1[$ of $t_1$ such that
$$\mbox{$ |\varphi(t)-\varphi(t_1)|<\e$ for $t\in I_\delta$}.$$
Then, for these $t$,
$$||\Phi(t)-\Phi(t_1)||= ||\varphi(t)\cdot \bs 1-\varphi(t_1)\cdot \bs 1||= |\varphi(t)-\varphi(t_1)|<\e.$$
Hence
$$|| e^{-h}\Phi(t)-{\bs 1}||=||e^{-h} (\Phi(t)-e^h)||\leq ||e^{-h}||\; \e <1.$$
By Lemma \ref{expo}, 
$e^{-h}\Phi(t)=e^g$ for some $g\in B$; choose for example
$$g=\sum_{n=1}^\infty (-1)^{n+1} \frac{1}{n} \left( e^{-h}\Phi(t)-\bs 1\right)^n.$$
Since $B$ is commutative, $\Phi(t)=e^h e^g=e^{h+g}\in \exp B$.
Thus $I_\delta\ss \mathcal O$.\\

{\bf Claim 2}  $\mathcal O$ is closed in $[0,1[$.

Let $t_n\in \mathcal O$ converge to some $s\in [0,1[$. By definition of $\mathcal O$, 
$\Phi(t_n)\in \exp B$; say  $\Phi(t_n)=e^{h_n}$ with $h_n\in B$. 
As in the previous paragraph, 
$$||\Phi(t_n)-\Phi(s)||\leq |\varphi(t_n)-\varphi(s)|\to 0.$$
 Since inversion is a continuous  operation in $A^{-1}$ and  $\Phi([0,1[)\ss A^{-1}$,   we conclude from
$||\Phi(t_n)-\Phi(s)||\to 0$, that $||\Phi(t_n)^{-1}-\Phi(s)^{-1}||\to 0$. In particular, for all $n$,
$$||e^{-h_n}||=||\Phi(t_n)^{-1}|| \leq M<\infty.$$ 
Now let $0<\e<1/M$.  Choose $n_0$ so that $|\varphi(t_n)-\varphi(s)|<\e$ for $n\geq n_0$.
Then for all $n\geq n_0$,
$$||e^{-h_n}\Phi(s) -{\bs 1}||= ||e^{-h_n}(\Phi(s)-e^{h_n})||\leq M \; ||\Phi(s)-e^{h_n}||$$
$$\leq M\;|\varphi(s)-\varphi(t_n)|\leq  M\; \e <1.$$
By Lemma \ref{expo} 
$e^{-h_{n_0}}\Phi(s)=e^k$ for some $k\in B$; choose for example
$$k=\sum_{n=1}^\infty (-1)^{n+1} \frac{1}{n} \left( e^{-h_{n_0}}\Phi(s)-\bs 1\right)^n.$$
Since $B$ is commutative,  $$\Phi(s)=e^{h_{n_0}}\,e^k= e^{h_{n_0}+k}.$$
Thus $s\in \mathcal O$ and so $\mathcal O$ is closed. 
Due to connectedness, we deduce that $\mathcal O=[0,1[$. Hence $a=\Phi(0)\in \exp B$.
{\phantom.}\hfill $\square$
\bigskip

Thus, we also obtain an easy proof (without the functional calculus or  the use of the Jordan decomposition) of the following well-known result:
\begin{corollary}
Let $M_n(\C)$ be the algebra of $n\times n$-matrices over $\C$. Then each invertible
 matrix $M$ is an exponential matrix.
\end{corollary}
\begin{proof}
Since the spectrum  of $M$ is finite, we see that the hypotheses of the previous theorem are all satisfied.
\end{proof}

 \section{Real Banach algebras}

Let $\mathcal R$ be  a real Banach algebra with unit element ${\bs 1}$. If $x\in \mathcal R$, then its (real-symmetric) spectrum
$\sigma^*_{\mathcal R}(x)$ is   defined as
\begin{eqnarray*}
\sigma^*_{\mathcal R}(x):&=&\mbox{$\{\lambda \in \C: (a-\lambda \cdot{\bs 1})\, (a-\ov\lambda \cdot{\bs 1})$ \;not invertible in $\mathcal R\}$}\\
&=&\mbox{$\{\lambda \in \C:a^2-2({\rm Re}\,\lambda)\, a +|\lambda|^2\,\cdot{\bs 1} $ \;not invertible in $\mathcal R\}$}.
\end{eqnarray*}
Note that $\sigma^*_{\mathcal R}(x)$ is a non-void compact set in $\C$ (see for example \cite{kuli}).
This definition, going back to Kaplansky, can be motivated by looking at the characteristic
polynomial of a real matrix $M$. Here the zeros of $$p_n(\lambda):=\det ( M-\lambda I_n)$$
 are either real or
appear in pairs $(\lambda,\ov\lambda)$. Hence $\det(M-\ov\lambda I_n)$ 
has the same zeros.
Since $M-\lambda I_n$ is  not a real matrix, one considers $N:=(M-\lambda I_n)(M-\ov \lambda I_n)$,
which has the property that $\det N=p_n(\lambda)^2$.  What we have gained is that
$N=M^2-2({\rm Re}\,\lambda)\, M +|\lambda|^2\; I_n $ is a real matrix again. Here is now the companion
result to Theorem  \ref{expoinba}.

\begin{theorem}\label{real-logs}
Let $\mathcal R$ be a unital real Banach algebra (not necessarily commutative) and 
let $x\in \mathcal R$.
\begin{enumerate}
\item [(1)]
Suppose that $0$ belongs to the unbounded connected component of 
$\C\setminus \sigma^*_{\mathcal R}(x)$. 
Then there exists $r\in \mathcal R$ such that $e^r=x^2$.
\item[(2)] 
If $]{-\infty}, 0]$ belongs  the unbounded connected component of 
$\C\setminus \sigma^*_{\mathcal R}(x)$, 
then there exists $v\in \mathcal R$ such that $e^v=x$.

\end{enumerate}
\end{theorem}
\begin{proof}
(1) We first note that the element
$u:=x^2-2({\rm Re}\,M)\, x +|M|^2\;\cdot{\bs 1} $ is an exponential in
 $\mathcal R$ whenever $M\in \C$
is large enough. In fact, since
$$f:=\frac{u}{|M|^2}= \left(\frac{x}{|M|}\right)^2 - 2\left({\rm Re}\,\frac{M}{|M|}\right)\, \frac{x}{|M|}+\bs 1\in \mathcal R$$
and satisfies $||f-\bs 1||<1$ for $M$ large enough, we deduce from Lemma \ref{expo}
that $f=e^L$ for some $L\in \mathcal R$. Hence $u=e^{2\log |M|\cdot {\bs 1}+L}$ is the desired
representation of $u$ (again because  $2\log |M|\cdot {\bs 1}$ and $L$ commute).

Now the proof works in exactly the same way as before; just
associate to the  arc $\varphi:[0,\infty[\to \C\setminus \sigma^*_{\mathcal R}(x)$ 
which joins $0$ to $\infty$, 
the path $\Phi:[0,\infty[\to \mathcal R^{-1}\inter [x]_{{\rm alg}}$
given by 
$$\Phi(t):=\big(x -\varphi(t)\cdot \bs 1\big)\,\big(x -\ov{\varphi(t)}\cdot \bs 1\big)=
x^2-2({\rm Re}\,\varphi(t))\, x +|\varphi(t)|^2\cdot {\bs 1} .
$$

(2) 
By assumption,  $\Phi(t)=x +t\cdot{\bs 1}\in \mathcal R^{-1}$ for every $t\in [0,\infty[$.
Moreover, for $M>0$ large,  
$$\Phi(M)= x +M\cdot{\bs 1}=M\left(\frac{x}{M} +{\bs 1}\right)$$
has a logarithm by Lemma \ref{expo} applied to $f:=\frac{x}{M} +{\bs 1}$.
The conclusion now follows  using  the previous arguments. 
\end{proof}

\begin{corollary}\label{asquare}

Let $\mathcal R$ be a unital real Banach algebra (not necessarily commutative) and 
let $x\in \mathcal R$. Suppose that $0$ belongs to the unbounded connected component of 
${\C\setminus\sigma^*_{\mathcal R}(x)}$. Then
$$\mbox{$x\in \exp {\mathcal R}$ if and only if $x=y^2$ for some $y\in \mathcal R$.}$$
\end{corollary}

\begin{proof}
If $x=e^b$ for some $b\in \mathcal R$, then  we choose $y=e^{b/2}$. Conversely, if $x=y^2$,
then, by the spectral mapping theorem  in its most primitive version, 
$$\sigma_{\mathcal R}^*(x)=\sigma_{\mathcal R}^*(y^2)=q(\sigma_{\mathcal R}^*(y)),$$ 
where $q(z)=z^2$ (see for example \cite{kuli}).
Let $\varphi:[0,\infty[\to \C\setminus \sigma^*_{\mathcal R}(x)$ be  a polygonial Jordan arc
joining $0$ with $\infty$. We may use polar-coordinates: so $\varphi(t)=r(t)e^{is(t)}$
for some continuous functions $r$ and $s$, $r>0$ on $]0,\infty[$ and $r(0)=0$. Then the arc
$\psi: [0,\infty[ \to \C$ given by $\psi(t):=\sqrt{r(t)} e^{is(t)/2}$ does not meet 
$\sigma_{\mathcal R}^*(y)$,
because otherwise $q(\psi(t_0))=\varphi(t_0)\in \sigma_{\mathcal R}^*(x)$. Hence, by Theorem 
\ref{real-logs} (1), $x=y^2\in \exp \mathcal R$.
\end{proof}

Let us emphasize that $x$ (or $-x$) itself does not necessarily have a real logarithm. 
Just look at the algebra $M_3(\R)$ of $3\times 3$ matrices  over $\R$ and
the diagonal matrix 
$$M=\left(\begin{matrix}  -1&0&0\\ 0&-2&0\\ 0&0& 1
\end{matrix}\right).
$$

Note that $\det M>0$. Suppose that there is $L\in M_3(\R)$ such that $M=e^L$.  At least one eigenvalue  $\lambda$
of $L$ is real. Let $x\in\R^3$  be a corresponding eigenvector.  Then $Lx=\lambda x$
and so $L^nx=\lambda^nx$ for all powers $n$. Hence $Mx=e^Lx=e^\lambda x$. But then the
positive real number $e^\lambda$ must be the eigenvalue $1$ of $M$; that is $\lambda=0$.
Since no other positive real eigenvalue of $M$ exists, the other two eigenvalues  of $L$
must be complex conjugate numbers $\mu$ and $\ov\mu$.  Again, if $y\in \C^3$ is an eigenvector
of $L$ with respect to $\mu$, then $M y=e^\mu y$ implies that $e^\mu \in \{-1,-2\}$. 
But $\ov y$ is the 
conjugate eigenvector of $L$ with respect to $\ov \mu$ and so $M \ov y=e^{\ov\mu}\ov y$.
It can't  be the case, though, that $\{e^\mu,e^{\ov\mu}\}\ss \{-1,2\}$.
Thus we achieved a contradiction. We conclude that $\pm M$ is not  a real exponential.  

It is interesting  to note that the matrix
$$\tilde M=\left(\begin{matrix}  -1&0&0\\ 0&-1&0\\ 0&0& 1
\end{matrix}\right)
$$
does have a logarithm in $M_3(\R)$. In fact, since $\tilde L^2=-\pi^2 \tilde L$, it is
straightforwad to check that
$$\tilde L=\left(\begin{matrix}  0&-\pi&0\\ \pi&0&0\\ 0&0& 0
\end{matrix}\right)
$$
is  a logarithm of $\tilde M$.

An explicit  description of those  matrices in $M_n(\R)$ having a real logarithm
is given in \cite{cu}.  It tells us that $M\in M_n(\R)$ with $\det M>0$ has a logarithm in 
$M_n(\R)$ if and only if every  Jordan block associated with a negative eigenvalue
of $M$ appears an even number of times. 

\section{More on real matrices}
Let  $\bs c_j$ be the canonical column vector
 $(0,\dots, 0, \buildrel1_{}^{^{j \atop \downarrow}},0, \dots, 0)^t $,
$$I_n= \left(\begin{matrix} 1 &&0 \\
 &\ddots&\\0&&1
\end{matrix}\right)= (\bs c_1,\dots,\bs c_n)$$
 the identity matrix and 
$$\tilde I_n:= \left(\begin{matrix} -1 &&&0 \\&1&&\\  &&\ddots&\\ 0&&&1
\end{matrix}\right)=(-\bs c_1, \bs c_2,\dots, \bs c_n).$$

It is well known that the group $G_{M_n(\R)}$ of invertible real  matrices has exactly
 two components:

\begin{proposition}\label{compos}

\begin{eqnarray*}
G^+&=&\{\;M\in G_{M_n(\R)}: \det M>0\}=\big\{\prod_{j=1}^m e^{M_j}: m\in \N^*, M_j\in M_n(\R)\big\}\\
G^-&=&\{M\in G_{M_n(\R)}: \det M<0\}=\big\{\tilde I_n \prod_{j=1}^m e^{M_j}: m\in \N^*, M_j\in M_n(\R)\big\}.
\end{eqnarray*}
\end{proposition}

 Below we shall  not only give a simple proof of this fact, but also obtain the result that every matrix in $G^+$ actually is  a product of at most {\it two} exponential matrices \footnote{ This is probably  known among the workers in the field, but we couldn't trace   it in the literature.}.
We will use a weak version of Schur's reduction theorem for matrices
(see for instance \cite{be} for a proof of the general case): 

 \begin{theorem}\label{schur2}

  If $M\in M_n(\R)$ and if  $M$ admits $k$ real eigenvalues
$\lambda_1,\dots, \lambda_k$, ($1\leq k\leq n)$, then there exists an orthogonal matrix $Q$
such that
$$Q^t M Q = \left(\begin{array}{ccc|ccc} \lambda_1 & & *&* &&*\\
&\ddots&&&&\\
0&&\lambda_k&*&&*\\
\hline0&&0&*&&*\\
\vdots&&\vdots&&&\\
0&&0&*&&*
\end{array}\right)
$$
\end{theorem}

\begin{theorem}
Every matrix $M$ in $M_n(\R)$ with $\det M>0$ is a product of two exponential matrices
in $M_n(\R)$.
\end{theorem}
\begin{proof}
Let $\lambda_1, \dots, \lambda_k$ be the negative eigenvalues of $M$ (listed according to their multiplicities). If there are none, Theorem \ref{real-logs} (2) implies the existence of a real logarithm. So in this case, the second exponential is $I_n$. Thus we may assume $k \geq 1$.

By Theorem \ref{schur2} there is an orthogonal matrix $Q\in M_n(\R)$ such that
$$A:=Q^t M Q = \left(\begin{array}{ccc|ccc} \lambda_1 & & *&* &&*\\
&\ddots&&&&\\
0&&\lambda_k&*&&*\\
\hline0&&0&*&&*\\
\vdots&&\vdots&&&\\
0&&0&*&&*
\end{array}\right).
$$
Since the characteristic polynomial 
$$p_M(\lambda)=\det(M-\lambda I_n)=(-1)^n \prod_{j=1}^n (\lambda-\lambda_j)$$
 evaluated at $\lambda=0$ equals
$\det M=\prod_{j=1}^n \lambda_j$,  and since non-real zeros  occur in pairs
 $(\lambda, \ov \lambda)$,  we conclude that the number of negative eigenvalues
 of $M$ is  even. Let $P$ be the diagonal matrix
 $$P=(-\bs c_1,\dots, -\bs c_k, \bs c_{k+1},\dots \bs c_n),$$
 where, as we know, $k\in 2\N$.
 Then 
$$P A= \left(\begin{array}{ccc|ccc} |\lambda_1| & & *&* &&*\\
&\ddots&&&&\\
0&&|\lambda_k|&*&&*\\
\hline0&&0&*&&*\\
\vdots&&\vdots&&&\\
0&&0&*&&*
\end{array}\right),
$$
and the eigenvalues of $PA$ are $|\lambda_1|,\dots, |\lambda_k|, \lambda_{k+1},\dots, \lambda _n$.
Note that the lower right block of $A$ did not change.   Thus $\sigma(PA)\;\inter \;]-\infty, 0]=\emp$.
By Theorem \ref{real-logs} (2),  $PA=e^B$ for some $B\in M_n(\R)$. Moreover, 
it is easy see that $P$ has a real logarithm, too, say $P=e^C$
(just use the block-matrix structure and the formula  
$$\exp \left( \begin{matrix} 0 &-t\\ t & 0 \end{matrix}\right)= \left(\begin{matrix} \cos t& -\sin t\\ \sin t& \cos t\end{matrix}\right)$$
and put $t=\pi$, respectively $t=0$).
Hence, by noticing that $Q^tQ=I_n$ and $PP=I_n$, 
$$M= Q A Q^t= Q(PPA)Q^t= Q(e^C e^B)Q^t=e^{QCQ^t} e^{QBQ^t}.$$

\end{proof}

\begin{corollary}
If $M\in M_n(\R)$ has a negative determinant, then there are $B_1,B_2\in M_n(\R)$
such that $M=\tilde I_n e^{B_1}e^{B_2}$.
\end{corollary}

\begin{proof}
Just consider the matrix $\tilde M:= \tilde I_n M$, for which  $\det \tilde M>0$ and use the previous
Theorem.
\end{proof}

Proposition \ref{compos} now immediately follows since $t\mapsto e^{tB_1}e^{tB_2}$
is a path in $G^+$  joining $I_n$ with $M$ (and similarily for $G^-$).
Just note that by Jacobi's trace formula $\det e^A=e^{{\rm trace}\,A}>0$ for $A\in M_n(\R)$.
\bigskip

As a corollary to Theorem \ref{real-logs} (or  as a special case to Corollary \ref{asquare}), we obtain the following  result:

\begin{corollary}\label{square}
A real matrix $M\in M_n(\R)$ has a logarithm if and only if $M=A^2$ for some $A\in M_n(\R)$.
\end{corollary}
\begin{proof}
If $M=e^B$ for some $B\in M_n(\R)$, then  $A:=e^{B/2}$ has the desired property that
$M=A^2$. On the other hand, if  $M=A^2$, then by Theorem \ref{real-logs}, $A^2$ has a logarithm
(note that the real-symmetric spectrum $\sigma^*_{M_n(\R)}(A)$ is finite and so it obviously
satisfies the assumptions of that Theorem).
\end{proof}

As we have seen, if $A\in M_n(\R)$ with $ \det A>0$, then $A=e^{B_1}e^{B_1}$ 
where $B_j\in M_n(\R)$,
but $A$ does not necessarily belong to $\exp M_n(\R)$. The following result, though, shows
that each $A\in G_{M_n(\R)} $ comes very close to admit a real logarithm.

\begin{corollary}\label{doppelblock}
 Let $A$ be an invertible matrix in $M_n(\R)$.
Then  the $(2n)\times (2n)$-matrix 
$$M:=\left(\begin{matrix} A& O\\ O& A\end{matrix}\right)$$
has  a logarithm in $M_{2n}(\R)$.
\end{corollary}
\begin{proof}
First we note that $\det M=(\det A)^2\not=0$. Hence $M$ is invertible in $M_{2n}(\R)$.
Let $H\in M_n(\C)$ be a complex logarithm of $A$. Put $C:= {\rm Re}\; e^{H/2}$ and 
$D:={\rm Im}\;e^{H/2}$.  Using Corollary \ref{square}, it suffices to show  that
$$\left(\begin{matrix} A& O\\ O& A\end{matrix}\right)=\left(\begin{matrix} C &  -D\\
D & C\end{matrix}\right)^2.$$
Note that
$$\left(\begin{matrix} C &  -D\\ D & C\end{matrix}\right)^2=\left(\begin{matrix} C^2-D^2 &  -CD-DC\\  DC+CD& C^2-D^2\end{matrix}\right).$$
Using that $C= (1/2)( e^{H/2} + e^{\ov H/2})$ and $ D= (-i/2) ( e^{H/2} - e^{\ov H/2})$,
we see that $CD+DC=0$ and 
$$C^2-D^2=(C+iD)^2= (e^{H/2})^2=e^H=A.\eqno\qedhere$$
\end{proof}

{\bf Acknowledgement.} We thank Amol Sasane for providing us with reference  \cite{cu}.

\bibliographystyle{amsplain}

\end{document}